\documentclass[12pt]{amsart}
\usepackage{a4}
\usepackage{amsthm, amsfonts, amssymb,latexsym}
\usepackage{enumerate, color}
\usepackage[english]{babel}
\usepackage[latin1]{inputenc}

\newtheorem{lemma}{Lemma}[section]

\newtheorem{theorem}{Theorem}[section]
\newtheorem{conjecture}{Conjecture}[section]

\newtheorem{cor}{Corollary}[section]

\newcommand{\eps}{\epsilon}
\numberwithin{equation}{section}

\newcommand{\beq}[1]{\begin{equation}\label{#1}}
\newcommand{\eeq}{\end{equation}}

\title[Extremal sum-product problems]{A bound on the multiplicative energy of a sum set and extremal sum-product problems}

\author[ O. Roche-Newton and D. Zhelezov]{Oliver Roche-Newton and Dmitry Zhelezov}

\address{O. Roche-Newton: Johann Radon Institute for Computational and Applied Mathematics (RICAM), Austrian Academy of Sciences, 4040 Linz, Austria }
\email{o.rochenewton@gmail.com }

\address{D. Zhelezov: Chalmers University of Technology and University of Gothenburg, 41296 Gothenburg, Sweden}
\email{zhelezov@chalmers.se}

\begin{document}

\begin{abstract}
In recent years some near-optimal estimates have been established for certain sum-product type estimates. This paper gives some first extremal results which provide information about when these bounds may or may not be tight. The main tool is a new result which provides a nontrivial upper bound on the multiplicative energy of a sum set or difference set.
\end{abstract}

\maketitle

\section{Introduction}

A now familiar variation on the Erd\H{o}s-Szemer\'{e}di sum-product problem is the following: given a set $X(A)$ which is defined via a combination of additive and multiplicative operations on an input set $A$, show that the set $X(A)$ is always ``large" compared to the original set $A$.

An example of a relatively old result of this type is due to Ungar \cite{ungar}, who showed that any finite\footnote{From now on, all sets are assumed to be finite unless stated otherwise.} set $P$ of points in the plane determines at least $|P|-1$ different pairwise directions, provided that the point set does not lie on a single line. If one then applies this bound in the case when $P=A \times A$, where $A$ is an arbitrary set of real numbers such that $|A| \geq 2$, it follows that
\begin{equation}
\left|\frac{A-A}{A-A}\right| \geq |A|^2-2,
\label{ungar}
\end{equation}
where
\begin{equation}
\frac{A-A}{A-A}:=\left\{\frac{a-b}{c-d}:a,b,c,d \in A, c\neq d\right\}.
\label{defn}
\end{equation}
This notation for defining sets in terms of additive and multiplicative operations will be used flexibly throughout the paper; for example $A(A+A):=\{a(b+c):a,b,c\in A\}$. Whenever the definition includes division, it is stipulated that we do not divide by zero, as in \eqref{defn}.

In recent years we have seen some more progress in this direction, thanks largely to progress in the field of discrete geometry. For example, it was recently established in \cite{BORN} that for any subset $A$ of the positive reals,
\begin{equation}
\left|\frac{A+A}{A+A}\right| \geq 2|A|^2-1.
\label{balog}
\end{equation}
It was also proven in \cite{BORN} that the same result holds when $A \subset \mathbb C$, albeit with a smaller and unspecified multiplicative constant in the place of $2$. Similarly, it was established in \cite{MORNS} that
\begin{equation}
|A(A+A+A+A)| \gg \frac{|A|^2}{\log |A|}
\label{ilya}
\end{equation}
holds for any set $A \subset \mathbb C$.\footnote{Throughout this paper, for positive values $X$ and $Y$ the notation $X \gg Y$ is used as a shorthand for $X\geq cY$, for some absolute constant $c>0$.}

Furthermore, building on the work of Guth and Katz \cite{GK}, it was established in \cite{rectangles} that
\begin{equation}
|(A-A)(A-A)|\gg \frac{|A|^2}{\log |A|},
\label{misha2}
\end{equation}
and the same argument also works to prove that $|(A+A)(A+A)|\gg |A|^2/\log |A|$.

In the form of inequalities \eqref{ungar}, \eqref{balog}, \eqref{ilya} and \eqref{misha2}, we are seeing some progress relating to the sum-product problem. In particular, each of these four inequalities is optimal up to constant and logarithmic factors, as can be seen by taking $A=\{1,2,\cdots,|A|\}$. Once we have optimal estimates for such problems, it seems natural to ask the question: ``under what circumstances can these results be tight?".

Each of \eqref{ungar}, \eqref{balog}, \eqref{ilya} and \eqref{misha2} can be close to tight in the more general case when $A$ has small sum set (i.e. when $|A+A| \leq c|A|$ for an absolute constant $c$). However, we are not aware of any other constructions which exhibit the tightness of these bounds. We make the following conjecture:

\begin{conjecture} \label{conj1} There exist absolute constants $c$ and $c'$ such that for any $A \subset \mathbb C$
$$\left|\frac{A+A}{A+A}\right| \leq c|A|^2 \Rightarrow |A+A| \leq c'|A|.$$
\end{conjecture}
Similar conjectures can be made for the sets $\frac{A-A}{A-A}$, $A(A+A+A+A)$ and $(A-A)(A-A)$. In this paper, we prove some weak first results in the direction of Conjecture \ref{conj1} and variants thereof. For example, we prove the following result:
\begin{theorem} \label{structure1}
There exists an absolute constant $C > 0$ such that if $A \subset \mathbb C$ satisfies ${\left|\frac{A+A}{A+A}\right| \ll |A|^2}$, then 
$$|A+A|\ll |A|^2\exp\left( -C\log^{1/3 - \eps}|A| \right).$$
\end{theorem}
The same result holds with any of $\frac{A-A}{A-A}$, $(A+A)(A+A)$ or $(A-A)(A-A)$ in the place of $\frac{A+A}{A+A}$. It is interesting to note that if we assume the following variation of the sum-product conjecture, namely that
\begin{equation} \label{eq:sum-product}
	\max\{|A+A|, |A/A|\} \gg |A|^{2-o(1)},
\end{equation}
we can get another conditional bound relevant to Conjecture \ref{conj1}. Indeed, let $|A+A| = |A|^{1+\delta}$ for some $\delta > 0$ and $B = A+A$. Then, by the Pl\"unnecke-Ruzsa inequality, 
$$
|B+B| \ll |A|^{1+4\delta}
$$
and
$$
|B/B| \ll |A|^2 
$$
by the hypothesis of the conjecture. Thus,  if $0 < \delta < 1/2$ the set $B$ violates (\ref{eq:sum-product}). In other words, if we assume that the full sum-product conjecture holds and ${\left|\frac{A+A}{A+A}\right| \ll |A|^2}$, then either $|A+A| \ll |A|$ or $|A+A| \gg |A|^{3/2 - o(1)}$. So in fact Theorem \ref{structure1} estimates the sum set size from the \emph{opposite} side to what one would expect from sum-product type estimates.
Unfortunately, the best exponent for the sum-product proved to date is $4/3 - o(1)$ (see \cite{solymosi}), which gives only trivial unconditional estimates for the above argument.

The new tool here which leads to Theorem \ref{structure1} and its variants is an upper bound on the multiplicative energy of a sum set or difference set. The \textit{multiplicative energy of} $A$, denoted $E^*(A)$, is the number of solutions to the equation
\begin{equation}
ab=cd,\,\,\,\,\,\,\,\ (a,b,c,d)\in A^4.
\label{energy}
\end{equation}

The notion of multiplicative energy has played a key role in several of the most recent works on the sum-product problem, most notably in the work of Solymosi \cite{solymosi} in bounding the multiplicative energy in terms of the sum set, and consequently deducing what stands as the best estimate towards the Erd\H{o}s-Szemer\'{e}di conjecture.

Since fixing $a,b$ and $c$ in \eqref{energy} determines\footnote{Here we are making the simplifying assumption that $0\notin A$. However, there are also at most $4|A|^2$ solutions to \eqref{energy} which include a zero somewhere, and so we can still write $E^*(A)\ll |A|^3$ if it is the case that $0\in A$.} $d$, a trivial upper bound is given by $E^*(A) \leq |A|^3$. Combining this with the trivial upper bound $|A+A|<|A|^2$, we note that a trivial upper bound for the multiplicative energy of a sum set is given by $E^*(A+A) \leq |A|^6$. We give a small improvement to this upper bound as follows.

\begin{theorem} \label{thm:sumenergy} For any $\eps > 0$ there are constants $C'(\eps), C''(\eps) > 0$ such that for any set $A \subset \mathbb C$
$$
E^*(A-A) \leq \max \left\{C''(\eps) |A|^{3 + \eps}, |A-A|^3\exp\left( -C'(\eps)\log^{1/3 - \eps}|A| \right) \right\}.
$$
In particular, there is a constant $C(\eps) > 0$ such that\footnote{In fact, one can replace $\eps$ with a $\log \log^{O(1)}$ factor and then the constant $C$ becomes absolute.}
$$
E^*(A-A) \ll |A|^6\exp\left( -C\log^{1/3 - \eps}|A| \right)
$$
\end{theorem}

Once Theorem \ref{thm:sumenergy} has been established, it is an easy corollary to show that the same bounds hold for the quantity $E^*(A+A)$.

The proof of Theorem \ref{thm:sumenergy} uses some heavy machinery from (additive) number theory, in the form of the Subspace Theorem, the Freiman Inverse Theorem and the Balog-Szemer\'edi-Gowers Theorem. The proof is based on work of Schwartz \cite{schwartz}, who used the Subspace Theorem to give estimates for a variant on the Erd\H{o}s unit distance problem. See also \cite{SS} and \cite{SSZ} for some similar combinatorial applications of the Subspace Theorem. Some similar results to Theorem \ref{thm:sumenergy} are proven in \cite{Z} with the roles of addition and multiplication reversed. Non-trivial bounds for the additive energy $E^+(AA)$ are established, although some extra conditions on $A$ are needed.

Once Theorem \ref{thm:sumenergy} has been established, all that is required is an application of the Cauchy-Schwarz inequality to prove some initial results in the direction of Conjecture \ref{conj1} as outlined above. 

It can be calculated that if $A$ is a geometric progression of the form $A=\{2^n:1\leq n \leq |A|\}$, then $E^*(A+A) \gg |A|^5$. The authors are unaware of any examples of sets $A$ for which $E^*(A+A)$ is significantly greater than $|A|^5$, and it would be interesting to close the gap between $|A|^5$ and $o(|A|^6)$ in either direction.

In section 2 we will give the proof of Theorem \ref{thm:sumenergy}, and give the details of the calculation which shows that it is possible that $E^*(A+A) \gg |A|^5$. In section 3, Theorem \ref{structure1} and its variants are deduced as a corollary of Theorem \ref{thm:sumenergy}, and we will also use an application of the Szemer\'{e}di-Trotter Theorem to prove a quantitatively improved version of Theorem \ref{structure1} with $A(A+A+A+A)$ in place of $\frac{A+A}{A+A}$.




\section{Multiplicative energy of sum and difference sets}

\subsection{Proof of Theorem \ref{thm:sumenergy}}

The key ingredient of the proof of Theorem \ref{thm:sumenergy} is the following lemma, which says that a difference set cannot contain a large number of elements from a multiplicative group of small rank. 
\begin{lemma} \label{lemma:sumset_multp_group}
	Let $\epsilon > 0$. Then there are positive constants $c(\epsilon), C(\eps)$ such that for any set $A$ of complex numbers the following holds. For any multiplicative group $\Gamma \subset \mathbb{C}^*$ of rank $c(\eps)\log |A|$, the number of pairs
	$$
		\{ (a_1,a_2) \in A \times A \,\,| \,\, a_1-a_2 \in \Gamma \}
	$$ 
is at most $C(\eps)|A|^{1+\epsilon}$.	
\end{lemma}

The proof of Lemma \ref{lemma:sumset_multp_group} is essentially an adaptation of the argument of Schwartz \cite{schwartz} and Schwartz, Solymosi, de Zeeuw \cite{SSZ} which they used in order to attack the Erd\H{o}s unit distance problem for configurations of points with a group structure. In turn, their argument relies on a powerful tool from number theory, namely the Subspace Theorem. It seems that it was first introduced to the field of arithmetic combinatorics by Chang in \cite{Chang}.  The Subspace Theorem has a few different formulations and the one we will use is due to Evertse, Schlickewei and Schmidt \cite{ESS} with the quantitative bounds due to Amoroso and Viada \cite{AV}. 

Suppose $a_1, \ldots,  a_k \in \mathbb{C}^*$ and 
$$
\Gamma = \{\alpha_1^{z_1} \cdots \alpha_r^{z_r}, z_i \in \mathbb{Z} \},
$$
so $\Gamma$ is a free multiplicative group\footnote{The original theorem is formulated in a more general setting, namely for the division group of $\Gamma$, but we will stick to the current formulation for simplicity.} of rank $r$. Consider the equation
\begin{equation} \label{eq:subspace_eq}
a_1x_1 + a_2x_2 + \cdots + a_kx_k = 1 
\end{equation}

with $a_i \in \mathbb{C}^*$ viewed as fixed coefficients and $x_i \in \Gamma$ as variables. A solution $(x_1, \ldots, x_k)$ to (\ref{eq:subspace_eq}) is called \emph{nondegenerate} if
for any non-empty $J \subsetneq \{1, \ldots, k \}$
$$
	\sum_{i \in J} a_ix_i \neq 0.
$$

\begin{theorem}[The Subspace Theorem]
The number $A(k, r)$ of nondegenerate  solutions to (\ref{eq:subspace_eq}) satisfies the bound
\begin{equation} \label{subspace_thm_ineq}
A(k, r) \leq {(8k)}^{4k^4(k + kr + 1)}.
\end{equation}
\end{theorem}

\begin{proof}{(of Lemma \ref{lemma:sumset_multp_group})}
	We start by constructing a graph $G$ with the vertex set identified with $A$ and placing an edge between two vertices $a_1$ and $a_2$ if and only if $a_1 - a_2 \in \Gamma$. Without loss of generality we assume that $-1 \in \Gamma$ and so if $x \in \Gamma$, $-x \in \Gamma$ as well. Thus, we consider $G$ as an unoriented graph. The number of elements in the set of edges  $E(G)$ is then half of the number of pairs $(a_1, a_2)$ such that $a_1 - a_2 \in \Gamma$. Our strategy is thus to show that $|E(G)|$ cannot be too large. 
	
	Indeed, any path of length $k$ between two vertices $b_1$ and $b_2$ in $G$, when rescaled, gives a solution to (\ref{eq:subspace_eq}), so the number of paths is controlled by the Subspace Theorem. The only problem to overcome is the degeneracy condition, which is addressed below. The calculations turn out to be similar to the proof of Theorem 2 in \cite{schwartz}, and essentially we merely observe that the proof in \cite{schwartz} does not utilise the hypothesis that the pairs are separated by distance $1$.  However, we decided to present the argument in full to make the proof more self-contained.

	We proceed in a few steps.
	
	Let $n = |A| = |V(G)|$ and $|E(G)| = n^{1+\epsilon}$ for some fixed $\epsilon > 0$. The rank $r$ of $\Gamma$ is $c(\eps) \log n$ with $c(\epsilon)$ to be defined in due course. We will show that provided $c(\eps)$ is small enough, $n$ cannot be arbitrarily large. 
	
\textbf {Step 1.} We prune $G$ so that the minimal degree $\delta(G)$ is at least $\frac{1}{2}n^{\epsilon}$. We simply remove one by one vertices with degree less than $
\frac{1}{2} n^{\eps}$. Since we can remove at most $\frac{1}{2} n^{1+\eps}$ edges, for the resulting graph $G'$ we have $\frac{1}{2} n^{1+\eps} \leq |E(G')|$. 
We reassign the label $G$ to this graph $G'$.

\textbf {Step 2.} Let $k\geq 1$ be an integer parameter to be chosen later. Any non-closed (but not necessarily simple) path $\langle v_1, \ldots, v_k \rangle$ in $G$ gives rise to the identity 
$$
	(v_1 - v_2) + \cdots + (v_{k-1} - v_k) = v_1 - v_k,
$$
and since $(v_i, v_{i+1}) \in E(G) $ we have $x_i := v_i - v_{i+1} \in \Gamma$. Then, dividing by $v_1 - v_k$ we get a solution to an equation of the type (\ref{eq:subspace_eq}).
However, in order to use the Subspace Theorem, we have to make sure such a solution is non-degenerate. Let us call a path non-degenerate if it gives a non-degenerate solution in the way we have just described. 

Let us fix a vertex $v$ and estimate the number of non-degenerate $k$-paths emanating from $v$. We build a path by adding each edge one by one while keeping the path non-degenerate. Since $\delta(G) \geq \frac{1}{2}n^{\eps}$, at each step we have at least that many edges to choose from, but for those making the path degenerate. The maximal number of such ``bad" edges after $l$ vertices have been chosen is $2^l - 1$, so the total number of non-degenerate $k$-paths $P_v$ emanating from $v$ is at least

$$
P_v \geq \prod^{k-1}_{l = 0} \left(\frac{1}{2}n^{\eps} - 2^l + 1\right) \geq \frac{n^{k\eps}}{2^{2k}},
$$provided $n^{\eps} \geq 2^{k+2}$. 
There are at most $n$ vertices in $V(G)$ and so by pigeonholing there is some vertex $w\in V(G)$ such that there are at least
 \begin{equation} \label{eq:path_num}
  \frac{n^{\eps k-1}}{2^{2k}}
 \end{equation}
non-degenerate $k$-paths between $v$ and $w$.	
		
\textbf {Step 3.} Now we can apply the Subspace Theorem. Combining the bounds (\ref{subspace_thm_ineq}) and (\ref{eq:path_num}) and taking logs, we have
\begin{equation} \label{eq:n_bound}
	(\eps k-1) \log n  \leq (4k^5 + 4k^5r + 4k^4)\log(8k) + k\log 4 < 5k^5r\log(8k) = c(\eps)5k^5\log(8k)\log n.
\end{equation}

In order for the last inequality in \eqref{eq:n_bound} to hold, it is sufficient to assume that $r > 16$. But we can take $k = \lceil \frac{2}{\eps} \rceil$  and $c(\eps) = \left( 5\log(8k)k^5 \right)^{-1}$, which implies that (\ref{eq:n_bound}) gives a contradiction. It must be the case that one of the assumptions $n^{\eps}\geq 2^{k+2}$ or $r>16$ does not hold. In other words, it must be the case that
$n < n_0(\eps)$ for some finite $n_0$. We finish the proof by taking the constant $C(\eps) = n_0$. \end{proof}

Another powerful and well-known tool is the (Freiman)-Green-Ruzsa theorem which describes sets with small doubling. It has long been known that sets with small doubling are very similar to generalised arithmetic progressions of bounded rank, but until recently the quantitative bounds were rather weak. However, in a series of breakthrough papers Schoen, Croot and Sisask, Sanders (see \cite{SchoenFreiman}, \cite{CS}, \cite{SanBR})\footnote{Of course, we are not even trying to make a comprehensive list of contributors. Instead, we refer the interested reader to \cite{SanExp}.}, to name a few, gradually improved the bounds to almost best possible. We refer the reader to \cite{SanExp} for an excellent exposition. 

We need some technical definitions in order to formulate the Freiman-Green-Ruzsa theorem for arbitrary abelian groups. Let $G$ be such an abelian group with the operation written additively.  A \emph{$d$-dimensional centred convex progression} $P \subset G$ is defined as an image of a symmetric convex body $Q \subset \mathbb{R}^d$ under a homomorphism $\phi : \mathbb{Z}^d \to G$, so that $\phi(\mathbb{Z}^d \cap Q) = P$.  

The quantitative version of the Freiman-Green-Ruzsa theorem with the state of the art bounds is formulated as follows (see Theorem 1.4 in \cite{SanExp} and references therein).

\begin{theorem} \label{FrGrRuzsa}
Suppose that $|A+A| \leq K|A|$. Then there is a set $X$, a finite subgroup $H$ and a $d$-dimensional centred convex progression $P$ such that
$$
A \subset X + H + P 
$$
and the following bounds\footnote{Here $o(1)$ is just a shortening for a $\log \log^{O(1)} K$ factor, rather than an asymptotic notation, see \cite{SanExp}.} hold
\begin{eqnarray}
|X| &\leq& \exp(C\log^{3+o(1)} K) \\
 d &\leq& C\log^{3+o(1)} K \\
 |H+P| &\leq& \exp(C\log^{3+o(1)} K)|A|.
\end{eqnarray}

\end{theorem}



In our case we are interested only in subsets lying inside subgroups of bounded rank, so we record the following corollary.
\begin{cor} \label{FrGrRuzsaCorr}
{\sloppy Let $A \subset \mathbb{C}^*$ with $|AA| < K|A|$. Then there is $A_1 \subset A$ such that 
$$
 |A|\exp(-C\log^{3+o(1)} K) \leq |A_1|
$$ and $A_1$ is contained in a multiplicative group of rank at most ${C\log^{3+o(1)} K}$},  where $C$ is a positive absolute constant.
\end{cor}
\begin{proof}
By Theorem \ref{FrGrRuzsa},
$$
A \subset X \cdot H \cdot P,
$$ 
where the same notation is used. Since $H$ is a finite subgroup, it is generated by a root of unity. Since $P = \phi(\mathbb{Z}^d \cap Q)$ for some homomorphism $\phi$, it is contained in the subgroup generated by $\phi(e_1),\cdots,\phi(e_d)$, where the $e_i$ are the standard basis elements. So, $P$ is contained in a subgroup of rank $d$. Finally, there is an $x \in X$ such that 
$|xH\cdot P \cap A| \geq |A|/|X|$ and since $\{ xH \cdot P \}$ is generated by at most $d+2$ elements, it follows that $A_1 = \{ xH \cdot P \} \cap A$ satisfies the desired conditions by Theorem \ref{FrGrRuzsa}.
\end{proof}

The last tool that will be needed for the proof of Theorem \ref{thm:sumenergy} is the following version of the Balog-Szemer\'edi-Gowers theorem due to Schoen \cite{SchoenBSG}. 
\begin{theorem} \label{thm:BSG}
Let $A$ be a subset of an abelian group, written additively, such that $E^{+}(A) = |A|^3/K$. Then there exists $A' \subset A$ such that $|A'| \gg |A|/K$ and 
$$
|A' - A'| \ll K^4|A'|.
$$
In particular, by the Pl\"unnecke-Ruzsa inequality, 
$$
|A' + A'| \ll K^8|A'|.
$$
\end{theorem}

Now we return to the proof of Theorem \ref{thm:sumenergy}.

\begin{proof}{(of Theorem \ref{thm:sumenergy})}

Let $B = A - A$ and $E^*(B) = |B|^3/K$. Applying the Balog-Szemer\'edi-Gowers theorem multiplicatively, we obtain the subset $B' \subset B$ such that $|B'| \gg |B|/K$ and $|B'B'| \ll K^8|B'|$. By Corollary \ref{FrGrRuzsaCorr}, there is $B'' \subset B'$ of size at least $|B'|\exp(-C\log^{3+o(1)} K)$ contained in a multiplicative group of rank $r$ at most $C\log^{3+o(1)} K$. 

Let us fix $1/100 > \eps > 0$, $c(\eps)$ and $C(\eps)$ given by Lemma \ref{lemma:sumset_multp_group} and denote these numbers by $c_1$ and $C_1$ respectively. Denote also $C'(\eps) = c_1/C$; it can be assumed that $C' \leq 1$, since otherwise we can take a sufficiently small value of $c_1(\epsilon)$ and Lemma \ref{lemma:sumset_multp_group} also holds for this value. Let us assume that 
$$
\log K \leq C'\log^{1/3 - \eps} |A|,
$$ so that in particular 
\begin{align*}
r &\leq C(C')^{3+o(1)} \log |A|
\\&\leq C C' \log |A|
\\&=c_1\log |A|.
\end{align*}
Then, we have $|B''| \leq C_1|A|^{1+\eps}$ by Lemma \ref{lemma:sumset_multp_group}. Expanding the inequalities, we have
\begin{align*} \label{eq:Kbound}
E^*(B) &=\frac{|B|^3}{K}
\\& \ll K^2|B'|^3 
\\& \leq K^2 \exp(3C\log^{3+o(1)} K) |B''|^3
\\&\leq  K^2 \exp(3C\log^{3+o(1)} K) C_1|A|^{3+3\eps}
\\&\leq \exp(2C' \log^{1/3-\eps} |A|) \exp(3CC' \log^{1-\eps/2}|A|) C_1 |A|^{3+3\eps}
\\& \leq C''(\eps)|A|^{3+4\eps},
\end{align*}
by our assumptions on $K$. Otherwise, if  $\log K \geq C'\log^{1/3 - \eps} |A|$, we have
\begin{equation} \label{eq:Ebound}
	E^*(B) \leq  |A-A|^3\exp(-C'\log^{1/3 - \eps}|A|).
\end{equation}
Taking $\eps$ small enough, we conclude that for arbitrarily small $\eps > 0$ there exist constants $C'(\eps),C''(\eps) > 0$ such that
\begin{equation} \label{eq:unify}
E^*(A-A) \leq \max \left\{ C''(\epsilon)|A|^{3 + \eps}, |A-A|^3\exp\left( -C'(\eps)\log^{1/3 - \eps}|A| \right) \right\}.
\end{equation}\end{proof}

The next task is to record a corollary of Theorem \ref{thm:sumenergy}, namely that the same result holds for $E^*(A+A)$.
\begin{theorem} \label{thm:sumenergy2} For any $\eps > 0$ there is a positive constant $C'(\eps)$ such that for any set $A \subset \mathbb C$
$$
E^*(A+A) \ll_{\eps} \max \left\{ |A|^{3 + \eps}, (|A-A|^3+|A+A|^3)\exp\left( -C'(\eps)\log^{1/3 - \eps}|A| \right) \right\}.
$$
In particular, there is a constant $C(\eps) > 0$ such that
$$
E^*(A+A) \ll |A|^6\exp\left( -C\log^{1/3 - \eps}|A| \right).
$$
\end{theorem}

\begin{proof} Write $B=A \cup -A$ and note that $A+A \subset B-B$. It follows from this inclusion that $E^*(A+A) \leq E^*(B-B)$. Now apply Theorem \ref{thm:sumenergy} for the set $B$. Since $|B| \leq 2|A|$ and $|B-B| \ll |A+A|+|A-A|$, the desired conclusion follows. \end{proof}


\subsection{A set whose sum set has large multiplicative energy}

To conclude this section, let us give more details of the calculation which shows that for $A=\{2^n:n \in \{1,\cdots,|A|\}\}$, we have $E^*(A+A) \gg |A|^5$. Let us assume for simplicity of exposition that $|A|$ is a multiple of 3.

The quantity $E^*(A+A)$ is the number of solutions to $s_1s_2=s_3s_4$ such that $s_i \in A+A$. However, since $A$ is a Sidon set, this is exactly the same as the number of solutions to
\begin{equation}
(2^{n_1}+2^{n_2})(2^{n_3}+2^{n_4})=(2^{n_5}+2^{n_6})(2^{n_7}+2^{n_8})
\label{energy1}
\end{equation}
such that $n_1,\cdots,n_8 \in \{1,\cdots,|A|\}$. After expanding the brackets \eqref{energy1} becomes
\begin{equation}
2^{n_1+n_3}+2^{n_1+n_4}+2^{n_2+n_3}+2^{n_2+n_4}=2^{n_5+n_7}+2^{n_5+n_8}+2^{n_6+n_7}+2^{n_6+n_8}.
\label{energy2}
\end{equation}
We will show that there are at least $\left(\frac{|A|}{3}\right)^5$ ``trivial" solutions to \eqref{energy2}, corresponding to octuples $(n_1,\cdots,n_8)\in [|A|]^8$ which satisfy the system of equations
\begin{align}
    n_1+n_3 &= n_5+n_7, \label{energy3}
\\ n_1+n_4&= n_5+n_8, \label{energy4}
\\ n_2+n_3&= n_6+n_7,  \label{energy5}
\\ n_2+n_4&= n_6+n_8.  \label{energy6}
\end{align}

To see this, simply choose any combination of five elements $n_1,\cdots,n_5$ from the middle third interval $\{\frac{|A|}{3},\frac{|A|}{3}+1\cdots,\frac{2|A|}{3}-1\}$. Then the octuple 
$$(n_1,n_2,n_3,n_4,n_5,n_2+n_5-n_1,n_1+n_3-n_5,n_1+n_4-n_5)\in [|A|]^8$$ 
satisfies the aforementioned system of equations. Indeed, \eqref{energy3} and \eqref{energy4} follow straight away from the choice of $n_7$ and $n_8$ in the above octuple. It remains to check that  \eqref{energy5} and \eqref{energy6} hold. This is indeed the case, since
\begin{align*}
    n_6 &= n_2+n_5-n_1, 
\\ &=n_2+n_3-(n_1+n_3-n_5)
\\ &=n_2+n_3-n_7
\end{align*}
and similarly
\begin{align*}
    n_6 &= n_2+n_5-n_1, 
\\ &=n_2+n_4-(n_1+n_4-n_5)
\\ &=n_2+n_4-n_8.
\end{align*}
The choices of $n_1,\cdots,n_5$ were made arbitrarily from a subset of $[|A|]$ of size $|A|/3$, and so it follows that we have at least $\left(\frac{|A|}{3}\right)^5$ solutions to \eqref{energy2}. This confirms that $E^*(A+A) \gg |A|^5$.

\section{Structural sum-product estimates}

\subsection{Proof of Theorem \ref{structure1}}

By the Cauchy-Schwarz inequality,
\begin{equation}
E^*(B)|B/B| \gg |B|^4
\label{CSclassic}
\end{equation}
 for any set $B \subset \mathbb C$. Applying this inequality with $B=A+A$, and then using the hypothesis that $\left|\frac{A+A}{A+A}\right| \ll |A|^2$, it follows that
$$|A+A|^4 \ll \left|\frac{A+A}{A+A}\right|E^*(A+A) \ll |A|^2E^*(A+A).$$
Applying Theorem \ref{thm:sumenergy2}, we have
$$|A+A|^4\ll |A|^8\exp(-C \log^{1/3-\eps}|A|),$$
for some constant $C$, from which the desired conclusion that $|A+A|\ll |A|^2\exp(-C' \log^{1/3-\eps}|A|)$ follows.
\flushright \qedsymbol \flushleft

It is straightforward to exchange $\frac{A+A}{A+A}$ with $(A+A)(A+A)$ in the above proof. To do this, simply use a slightly different version of \eqref{CSclassic} in the form of
$$E^*(B)|BB| \gg |B|^4.$$
Furthermore, it is a straightforward task to modify the proof of Theorem \ref{structure1} by taking $B=A-A$  and using Theorem \ref{thm:sumenergy} instead of Theorem \ref{thm:sumenergy2}, in order to obtain matching result for products and ratios of difference sets. We summarise these remarks as well as Theorem \ref{structure1} in the following theorem:
\begin{theorem}There exists a positive absolute constant $C$ such that for any finite set $A\subset \mathbb C$, each of the following statements holds:
\begin{enumerate}
\item if $\left|\frac{A+A}{A+A}\right| \ll |A|^2$ then $|A+A|\ll |A|^2\exp(-C' \log^{1/3-\eps}|A|)$;
\item if $|(A+A)(A+A)| \ll |A|^2$ then $|A+A|\ll |A|^2\exp(-C' \log^{1/3-\eps}|A|)$;
\item if $\left|\frac{A-A}{A-A}\right| \ll |A|^2$ then $|A-A|\ll |A|^2\exp(-C' \log^{1/3-\eps}|A|)$;
\item if $|(A-A)(A-A)| \ll |A|^2$ then $|A-A|\ll |A|^2\exp(-C' \log^{1/3-\eps}|A|)$.
\end{enumerate}
\end{theorem}

\subsection{Using the Szemer\'{e}di-Trotter Theorem}

In this subsection, the Szemer\'{e}di-Trotter Theorem will be used to prove a version of Theorem \ref{structure1} for the set $A(A+A+A+A)$. Recall that the Szemer\'{e}di-Trotter Theorem is the following:

\begin{theorem}[The Szemer\'{e}di-Trotter Theorem] \label{thm:ST}
Given a set $P$ of points and a family $L$ of lines in $\mathbb R^2$, the set of incidences $I(P,L):=\{(p,l)\in P\times L:p \in l\}$ satisfies
$$|I(P,L)| \ll |P|^{2/3}|L|^{2/3} +|P|+|L|.$$
\end{theorem}

The Szemer\'{e}di-Trotter Theorem has a long association with the sum-product problem. The relationship was initiated by the work of Elekes \cite{elekes}, who used Theorem \ref{thm:ST} to give a short proof that $\max \{|A+A|,|AA|\} \gg |A|^{5/4}$. Since then, incidence geometry has been used as a tool in proving several sum-product inequalities, including the proof of \eqref{ilya}.

We will use Theorem \ref{thm:ST} to prove the following lemma. Although the proof is standard, and similar result have been referenced in the literature, it seems that the full statement and proof have not been published in their entirety, and so a proof is given here for completeness. A version of this result with the roles of multiplication and addition reversed was set as Exercise 8.3.3 in \cite{tv}.

\begin{lemma} \label{thm:ST2}
Let $A \neq \{0\}$ be a subset of $\mathbb R$ and let $B$ and $C$ be any sets of real numbers such that $B+C \neq \{0\}$. Then
$$|A(B+C)| \gg  (|A||B||C|)^{1/2}.$$
\end{lemma}

\begin{proof} Let $A^*=A \setminus \{0\}$, and note that it follows from the assumption on the set $A$ that $A^*$ is non empty, and thus $|A^*| \gg |A|$. Let $l_{a,b}$ denote the line with equation $y=a(x+b)$, and let
$$L:=\{l_{a,b}:a,b \in A^* \times B\}.$$
By ensuring that $a$ is non-zero, we have $|L|=|A^*||B| \gg |A||B|$. Let $P=C \times A(B+C)$, and note that each line $l_{a,b} \in L$ has at least $|C|$ incidences with $P$, since for all $c\in C$,
$$(c,a(b+c))\in P \cap l_{a,b}.$$
Therefore, $|I(P,L)| \gg |A||B||C|$, and then by the Szemer\'{e}di-Trotter theorem
\begin{equation}
\label{STstep}
|A||B||C| \ll (|A||B||C||A(B+C)|)^{2/3}+|C||A(B+C)|+|A||B|.
\end{equation}
If $|C| \ll 1$ then $|A(B+C) \geq |A(B+c)| \gg |A|^{1/2}|B|^{1/2} \gg (|A||B||C|)^{1/2}$ for some\footnote{The inequality $A(B+c) \gg |A|^{1/2}|B|^{1/2}$ uses the non-degeneracy condition that $B+C \neq \{0\}$ and therefore there is some $c \in C$ such that $B+c \neq \{0\}$.} $c\in C$. Therefore, we can assume that the third term in \eqref{STstep} is irrelevant and thus
\begin{equation}
\label{STstep2}
|A||B||C| \ll (|A||B||C||A(B+C)|)^{2/3}+|C||A(B+C)|.
\end{equation}
Depending on which of the two terms on the RHS of \eqref{STstep2} is dominant, we have either $|A(B+C)| \gg (|A||B||C|)^{1/2}$ or $|A(B+C)| \gg |A||B|$, that is,
\begin{equation}
|A(B+C)| \gg \min  \{(|A||B||C|)^{1/2},|A||B|\}.
\label{nearly}
\end{equation}
Finally, one can interchange the roles of $B$ and $C$ in the proof of \eqref{nearly} to establish that
\begin{equation}
|A(B+C)| \gg \min  \{(|A||B||C|)^{1/2},|A||C|\},
\label{nearly2}
\end{equation}
and since $(|A||B||C|)^{1/2}$ cannot dominate both $|A||B|$ and $|A||C|$, the proof is complete.
\end{proof}

We are now ready to use this lemma to prove a structural result for $|A(A+A+A+A)|$.

\begin{theorem} \label{structure12}
If $|A(A+A+A+A)| \ll |A|^2$ then $|A+A| \ll |A|^{3/2}$.
\end{theorem}

\begin{proof} Applying Lemma \ref{thm:ST2} with $B=C=A+A$, as well as the assumption that $|A(A+A+A+A)| \ll |A|^2$, it follows that
$$|A+A||A|^{1/2} \ll |A(A+A+A+A)|\ll |A|^2.$$
Simply rearranging this gives $|A+A| \ll |A|^{3/2}$ as claimed.
\end{proof}

\textbf{Remark.} In fact, is easy to modify the proof of Theorem \ref{structure12} to get the same result with the set $(A+A)(A+A+A)$ in the place of $A(A+A+A+A)$.\footnote{We are grateful to Ilya Shkredov for bringing this to our attention.}

\section{Acknowledgements}Oliver Roche-Newton was supported by the Austrian Science Fund (FWF): Project F5511-N26, which is part of the Special Research Program ``Quasi-Monte Carlo Methods: Theory and Applications". Part of this research was undertaken when the authors were visiting the Institute for Pure and Applied Mathematics, UCLA, which is funded by the NSF. We are grateful to Brandon Hanson, Nets Katz, Ilya Shkredov and Jozsef Solymosi for several helpful conversations related to the content of this paper.

\end{document}